\documentclass[11pt]{amsart}

\usepackage{amsmath}
\usepackage{mathabx}

\usepackage{bbm}
\usepackage{enumerate}
\usepackage{paralist}
\usepackage{amssymb}

\usepackage{stmaryrd}
\usepackage{mathrsfs}

\newtheorem{theorem}{Theorem}

\newtheorem{lemma}[theorem]{Lemma}
\newtheorem{corollary}[theorem]{Corollary}

\theoremstyle{definition}

\theoremstyle{remark}

\newcommand \supp{ \mbox{supp}}

\newcommand \SI{ S_\infty}
\newcommand \Exh{ \mbox{Exh}}
\newcommand \Fin{ \mbox{Fin}}
\newcommand \res{\upharpoonright}
\newcommand \Sym{ \mbox{Sym}}

\newcommand \NN{\mathbbm{N}}

\begin{document}

\title{An example of a non non-archimedean Polish group with ample generics}
\author{Maciej Malicki}

\address{Department of Mathematics and Mathematical Economics, Warsaw School of Economics, al. Niepodleglosci 162, 02-554,Warsaw, Poland}
\email{mamalicki@gmail.com}
\date{March 04, 2015}
\keywords{ample generics, non-archimedean groups, P-ideals}
\subjclass[2010]{03E15, 54H11}
 
\begin{abstract}
For an analytic $P$-ideal $I$, $S_I$ is the Polish group of all permutations of $\NN$ whose support is in $I$, with Polish topology given by the corresponding submeasure on $I$. We show that if $\Fin \subsetneq I$, then $S_I$ has ample generics. This implies that there exists a non non-archimedean Polish group with ample generics.
\end{abstract}

\maketitle

\section{Introduction}

A Polish (i.e. separable and completely metrizable) topological group $G$ has \emph{ample generics} if the diagonal action of $G$ on $G^{n+1}$ by conjugation has a comeager orbit for every $n \in \NN$. This notion was introduced by W. Hodges, I. Hodkinson, D. Lascar and S. Shelah \cite{HoHo}, and in recent years it has drawn attention of many researchers (see \cite{KeRo} for more details.) An important motivation behind these investigations is that the existence of ample generics entails very interesting and strong consequences: a Polish group $G$ with ample generics has the automatic continuity property (i.e. every homomorphism from $G$ into a separable group is continuous), the small index property (i.e. every subgroup $H \leq G$ with $[G:H]<2^\omega$ is open), and uncountable cofinality for non-open subgroups (i.e. every countable exhaustive chain of non-open subgroups of $G$ is finite.) Moreover, by the general theory of Polish groups, the automatic continuity property implies that there exists a unique Polish group topology on $G$.

One of fundamental results in this area (see \cite{KeRo}) provides a complete characterization of Polish groups with ample generics that are subgroups of the group $S_\infty$ of all permutations of the natural numbers, i.e. \emph{non-archimedean} groups. As a matter of fact, all the known so far Polish groups with ample generics are of this form, and, as A. Kechris \cite{Ke} put it, it is `an important open problem (...) whether there exist Polish groups that fail to be non-archimedean but have ample generics'. In this note, we solve this problem by indicating a whole family of Polish groups with ample generics that are non non-archimedean.

\section{Results}
Recall that a \emph{lower semi-continuous submeasure} on $\NN$ is a function $\phi: \mathcal{P}(\NN) \rightarrow [0,\infty]$ satisfying

\begin{itemize}
\item $\phi(\emptyset) = 0$,
\item $A \subseteq B$ implies that $\phi(A) \leq \phi(B)$,
\item $\phi(A \cup B) \leq \phi(A) +\phi(B)$, and $\phi(\{n\}) < \infty$ for $n \in \NN$,
\item $\phi(\bigcup_m A_m) = \lim_m \phi(A_m)$ whenever $A_0 \subseteq A_1 \subseteq \ldots$.
\end{itemize}

Let $I$ be an analytic $P$-ideal on $\NN$ containing the ideal $\Fin$ consisting of finite sets. It is well known (see \cite{Sol1}) that there exists a lower semi-continuous submeasure $\phi$ on $\NN$ such that $I=\Exh(\phi)$, where
\[ \Exh(\phi) = \{A \subseteq \NN : \lim_m \phi(A \setminus [0,m]) = 0 \}. \]

As in \cite{Ts}, we associate with $I$ the group $S_I \leq \SI$ of permutations of $\NN$ defined by
\[ S_I=\{ g \in \SI: \supp(g) \in I \}. \]

Then (see \cite[Theorem 5.3]{Ts}) $S_I$ is Polishable, and its Polish group topology $\tau_I$ is given by the metric
\[ d(f,g)=\phi(\{f \neq g\}). \]

\begin{theorem}
\label{thMain}
Let $I$ be an analytic $P$-ideal such that $\Fin \subsetneq I$. Then $S_I$ has ample generics.
\end{theorem}

The following fact is well known. We prove it for the sake of completeness.

\begin{lemma}
\label{le:fo}
Let $X$ be a Polish space, and let $G$ be a Polish group continuously acting on $X$. Suppose that $x \in X$ is such that for every open neighborhood of the identity $V \subseteq G$, the set $V.x$ is somewhere dense in $\overline{G.x}$. Then the orbit of  $x$ is comeager in $\overline{G.x}$.
\end{lemma}

\begin{proof}
Suppose that $G.x$ is not comeager in $\overline{G.x}$. Since $G.x$ is analytic, and so it has the Baire property in $\overline{G.x}$, there exists an open set $U \subseteq X$ such that $G.x \cap U$ is non-empty and meager in $\overline{G.x}$. In other words, $U \cap G.x \subseteq \bigcup_n F_n$, where each $F_n$ is closed and nowhere dense in $\overline{G.x}$. The set
\[ W=\{ g \in G: g.x \in U \} \]
is open, non-empty in $G$, so there exists a non-meager $W' \subseteq W$ such that $W'.x \subseteq F_{n_0}$ for some $n_0$. By continuity of the action, $\overline{W'}.x \subseteq F_{n_0}$ as well. Since $\overline{W'}$ has non-empty interior, there exists $g \in G$ and an open neighborhood of the identity $V$ such that $gV.x \subseteq F_{n_0}$. Therefore $V.x \subseteq g^{-1}F_{n_0}$, and $V.x$ is nowhere dense in $\overline{G.x}$; a contradiction.
\end{proof}

\begin{lemma}
\label{le1}
Let $I$ be an analytic $P$-ideal, and let $f,g \in S_I$. Suppose that $A \in I$, and $h_m \in S_I$, $m \in \NN$, are such that $\supp (h_m) \subseteq A$, and 
\[ f \res [0,m] =h_m g h_m^{-1} \res [0,m] \]
for all $m$. Then $h_m g h_m^{-1} \rightarrow f$.
\end{lemma}

\begin{proof}
Let $\phi$ be a lower semi-continuous submeasure on $\NN$ such that $I=\Exh(\phi)$. Put $B= A  \cup \supp(f) \cup \supp(g)$. Clearly, $B \in I$, and 
\[ \supp(h_m g h_m^{-1}) \subseteq \supp(h_m) \cup \supp(g) \subseteq B. \]
Since $\phi(B \setminus [0,m]) \rightarrow 0$, and, by our assumption,
\[ \{ f \neq h_m g h_m^{-1} \} \subseteq B \setminus [0,m] \]
for every $m$, we get that $d(f,h_m g h_m^{-1}) \rightarrow 0$, i.e. $h_m g h_m^{-1} \rightarrow f$.
\end{proof}
 

\begin{proof}[Proof of Theorem \ref{thMain}]
Let $\phi$ be a lower semi-continuous submeasure on $\NN$ such that $I=\Exh(\phi)$. We can assume that $\phi(\{n\})>0$ for all $n \in \NN$. For $\epsilon>0$, let $V_\epsilon \subseteq S_I$ be the neighbourhood of the identity in $S_I$ of the form
\[ V_\epsilon=\{f \in S_I: \phi(\supp(f))<\epsilon \}. \]

%
%

Fix $n \in \NN$. We will show that there exist $g_0, \ldots, g_n \in S_I$ such that for every $\epsilon>0$ the set $V_\epsilon.(g_0, \ldots, g_n)$ is somewhere dense in $S^{n+1}_I$, and $S_I.(g_0, \ldots, g_n)$ is dense in $S^{n+1}_I$, where 
\[ g.(g_0, \ldots, g_n)=(gg_0g^{-1}, \ldots, gg_ng^{-1}), \]
for $g, g_0, \ldots, g_n \in S_I$. As $n$ and $\epsilon$ are arbitrary, by Lemma \ref{le:fo}, this will imply that $S_I$ has ample generics.

Fix an infinite $A \in I$. It is easy to find $g_0, \ldots, g_n \in S_I$ such that the following conditions are satisfied:

\begin{enumerate}[i)]
\item $\supp (g_i) \subseteq A$ for $i \leq n$,
\item for every $m$ there exists $m_0>m$ such that $[0,m_0]$ is invariant under the action of each $g_i$, 
\item for every $m \in \NN$, finite $B_0 \subseteq \NN$, and every $f_0, \ldots, f_n \in \Sym(B_0)$ there exists $A_0 \subseteq A$ with $\min A_0>m$, and a bijection $h:A_0 \rightarrow B_0$ such that
\[ hg_i h^{-1}(b)=f_i(b) \]
for all $b \in B_0$, and $i \leq n$.
\end{enumerate}

We show that $g_0, \ldots, g_n$ are as required. Fix $\epsilon>0$. Fix $m_0 \in \NN$, and $\epsilon'>0$ such that
\begin{enumerate}[a)]
\item $\phi(A \setminus [0, m_0])<\epsilon/2$,
\item $[0,m_0]$ is invariant under the action of each $g_i$,
\item $(n+1)\epsilon'<\epsilon/2$,
\item $d(f,g_i)<\epsilon'$ implies that $f \res [0,m_0] =g_i \res [0,m_0]$, $i \leq n$.
\end{enumerate}

Fix $f_0, \ldots,f_n \in S_I$ such that $d(f_i,g_i)<\epsilon'$, $i \leq n$, and each $f_i$ has finite support. Let $B=\bigcup_{i \leq n} \supp(f_i)$, and let $m_1=\max B$. For every $m>m_0,m_1$ we will construct $h_m \in S_I$ such that
\[ \supp(h_m) \subseteq  (A \cup B) \setminus [0,m_0], \mbox{ and } d(f_i,h_m g_i h_m^{-1}) \rightarrow 0 \]
for every $i \leq n$. Observe that the former, together with Points a) and c), implies that
\[ \phi(\supp(h_m))<\epsilon/2+(n+1)\epsilon'<\epsilon, \]
i.e. $h_m \in V_\epsilon$.

Fix $m>m_0, m_1$. Put $B_0=(A \cup B) \cap (m_0,m]$. By Points b) and d), $f_i \res B_0 \in \Sym(B_0)$ for each $i \leq n$. By Point iii), there exists $A_0 \subseteq A$ with $\min A_0>m$, and a bijection $h:A_0 \rightarrow B_0$ such that
\[ hg_i h^{-1} \res B_0= f_i \res B_0 \]
for each $i \leq n$. Clearly, we can extend $h$ to a permutation $h_m$ of $\NN$ such that $\supp(h_m) \subseteq A_0 \cup B_0 \subseteq (A \cup B) \setminus [0,m_0]$. But then, by our choice of $f_i$ and Point d), we get that
\[ hg_i h^{-1} \res [0,m]=f_i \res [0,m]. \]
By Lemma \ref{le1}, $h_m g_i h_m^{-1} \rightarrow f_i$ for each $i \leq n$.

Since permutations with finite support are dense in $S_I$, and the only other requirement we imposed on $f_i$ is that $d(f_i,g_i)<\epsilon'$, the above shows that $V_\epsilon.(g_0, \ldots, g_n)$ is somewhere dense. Observe that if consider $S_I$ instead of $V_\epsilon$ (and set $m_0=-1$), the same argument gives that the orbit $S_I.(g_0, \ldots, g_n)$ is dense in $S_I^{n+1}$.
\end{proof}  

Recall that an ideal $I$ is called a \emph{trivial modification of} $\Fin$ if there exists $A \subseteq I$ such that 
\[ I=\{ B \subseteq \NN: A \cap B \in \Fin \}. \]

\begin{corollary}
Suppose that $I$ is a $\boldsymbol{\Sigma}^0_2$ $P$-ideal which contains $\Fin$, and is not a trivial modification of $\Fin$. Then $S_I$ is a non non-archimedean Polish group with ample generics. In particular, there exists such a group, e.g. $S_{I_S}$, where $I_S$ is the summable ideal.
\end{corollary}

\begin{proof}
By \cite{Sol2}, $I$ is zero-dimensional in the Polish topology $\tau_I$ if and only if $I$ is a trivial modification of $\Fin$. Moreover, by \cite[Theorem 5.3]{Ts}, $S_I$ is zero-dimensional if and only if $I$ is zero-dimensional, so $S_I$ is not zero-dimensional. But a non-archimedean group must be zero-dimensional, so $S_I$ is non non-archimedean. By Theorem \ref{thMain}, $S_I$ has ample generics.

Clearly, the summable ideal 
\[ I_S= \{ A \subseteq \NN: \sum_{n \in A} 1/n<\infty \} \]
is a $\boldsymbol{\Sigma}^0_2$ $P$-ideal which is not a trivial modification of $\Fin$.  
\end{proof}

\end{document}